\documentclass[11pt,a4paper]{amsart}
\usepackage{amsmath,amsfonts,amsthm,amsopn,color,amssymb,enumitem}
\usepackage{palatino}
\usepackage{graphicx}
\usepackage[english]{babel}
\usepackage{caption}
\usepackage{subcaption}
\usepackage[colorlinks=true]{hyperref}
\hypersetup{urlcolor=blue, citecolor=red, linkcolor=blue}

\usepackage{cite}
\usepackage{esint}
\usepackage[title]{appendix}

\newtheorem{theorem}{Theorem}[section]
\newtheorem{lemma}[theorem]{Lemma}

\newtheorem{proposition}[theorem]{Proposition}

\newcommand{\e}{\varepsilon}
\newcommand{\C}{\mathbb{C}}
\newcommand{\R}{\mathbb{R}}
\newcommand{\Z}{\mathbb{Z}}

\newcommand{\weakto}{\rightharpoonup}

\newcommand{\T}{\mathbb{T}}

\newcommand{\N}{\mathbb{N}}

\renewcommand{\C}{\mathbb{C}}

\def\bbm[#1]{\mbox{\boldmath $#1$}}
\newcommand{\beq }{\begin{equation}}
\newcommand{\eeq }{\end{equation}}

\def\sideremark#1{\ifvmode\leavevmode\fi\vadjust{\vbox to0pt{\vss% the remark3
 \hbox to 0pt{\hskip\hsize\hskip1em%                          will appear only
 \vbox{\hsize3cm\tiny\raggedright\pretolerance10000%          on the side
  \noindent #1\hfill}\hss}\vbox to8pt{\vfil}\vss}}}%

\begin{document}

\title[traveling waves for the Gross-Pitaevskii equation in tori]{Existence and nonexistence of traveling waves for the Gross-Pitaevskii equation in tori}

\author{Francisco Javier Mart\'{i}nez S\'{a}nchez}
       \address{Francisco Javier Martinez S\'{a}nchez}
 \email{javims96@correo.ugr.es}

\author{David Ruiz}
  \address{David Ruiz \\
    IMAG, Universidad de Granada\\
    Departamento de An\'alisis Matem\'atico\\
    Campus Fuentenueva\\
    18071 Granada, Spain}
  \email{daruiz@ugr.es}

\thanks{D. R. has been supported by the FEDER-MINECO Grant PGC2018-096422-B-I00 and by J. Andalucia (FQM-116). He also acknowledges financial support from the Spanish Ministry of Science and Innovation (MICINN), through the \emph{IMAG-Maria de Maeztu} Excellence  Grant CEX2020-001105-M/AEI/10.13039/501100011033.}

% Type down your paper title

%\keywords{Prescribed curvature problem, conformal metric, Leray-Schauder degree.}

% \subjclass[2000]{35J20, 58J32, 53A30, 35B44}

% The abstract

\maketitle

\begin{center} \emph{On June 29, 2020, Ireneo posted: ``I miss my normal life''. We do miss you having your normal life, my friend.}
	
\end{center}

\begin{abstract}
	In this paper we consider traveling waves for the Gross-Pitaevskii equation which are $T$-periodic in each variable. We prove that if $T$ is large enough, there exists a solution as a global minimizer of the corresponding action functional. In the subsonic case, we can use variational methods to prove the existence of a mountain-pass solution. 
	Moreover, we show that for small $T$ the problem admits only constant solutions.  %Moreover, we show that in the planar case the mountain pass solution obtained are not 1-dimensional for almost all values of the period $T$.
\end{abstract}

\section{Introduction}
In this paper we are concerned with the Gross-Pitaevskii equation

\begin{equation}\label{GP}
i \partial_t \Psi=\Delta \Psi+\Psi\left(1-|\Psi|^2\right)  \text{on } \R^N \times \R.
\end{equation}
Here $\Psi$ is the wave function and $N=2$ or $3$ is the spatial dimension. This is a Nonlinear Schr\"{o}dinger Equation under the effect of a Ginzburg-Landau potential. The Gross-Pitaevskii equation was proposed in 1961 (\cite{gross, pita}) to model a quantum system of bosons in a Bose-Einstein condensate, via a Hartree-Fock approximation (see also \cite{b1, b2, JPR0, JPR}). It appears also in other contexts such as the study of dark solitons in nonlinear optics (\cite{k1, k2}).

From the point of view of the dynamics, the Cauchy problem for the Gross-Pitaevskii equation  was first studied in one space dimension by Zhidkov \cite{Z} and in dimension $N=2,3$
by B\'{e}thuel and Saut \cite{1999} (see also \cite{ge1,ge2, killip}). At least formally, equation \eqref{GP} presents two invariants, namely:
\begin{itemize}
	\item \emph{Energy:} 
	\[
	E= \int \frac 12 |\nabla \Psi|^2 +\frac 14  \left(1-|\Psi|^2\right)^2,
	\]
	\item \emph{Momentum:} 
	\[
	\mathcal{\bf{P}}=\frac 12 \int (i \nabla \Psi) \cdot  \Psi, 
	\]
	where $ f \cdot g =Re(f)Re(g)+Im(f)Im(g)$. For later use the first component of the momentum will be of special interest:
	\[
	P=\frac 12 \int  (i \partial_{x_1}\Psi) \cdot  \Psi.
	\]
	
\end{itemize}

In this work, we are concerned with the existence of periodic traveling waves of (\ref{GP}). Traveling waves for (\ref{GP}) are special solutions to (\ref{GP}) of the form 
\begin{equation}\label{eq:ansatz}
\Psi(x,t)=\psi(x_1-ct, \tilde{x}), \ \ \tilde{x}=(x_2 \dots x_N) \in \R^{N-1},
\end{equation}
%\begin{equation}\label{eq:ansatz}
%\Psi(x,t)=\psi(x_1-ct, x_2),
%\end{equation}
where the parameter $c >0$ characterizes the speed of the traveling wave and $x_1$ indicates the direction of the wave. By the ansatz \eqref{eq:ansatz} the equation for the profile $\psi$ is given by
\begin{equation}\tag{TWc}\label{TWc}
ic \partial_{x_1} \psi + \Delta \psi + (1-\vert \psi\vert^2)\psi  = 0.
\end{equation}

The case of solutions $\psi: \R^N \to \C$ with finite energy has attracted a lot of attention in the literature. The existence, nonexistence and qualitative behavior has been very much studied as part of the so-called Jones-Putterman-Roberts program. In particular, in \cite{JPR0, JPR} it was conjectured that such solutions exist only if $c \in (0, \sqrt{2})$. The value $c= \sqrt{2}$ is interpreted as the speed of sound, and is related to the behavior of the linearization around the constant solutions of modulus 1. Indeed, finite energy traveling waves for supersonic speed $c>\sqrt{2}$ are constant, see \cite{gravejat-CMP}. In dimension $N=2$ this result holds also for $c=\sqrt{2}$, see \cite{gravejat-DIA}. 

For small $c>0$ existence of solutions were proved in \cite{1999}, see also \cite{ao, NA2004, chr2, chr3, jer2, lw, lwy} for its asymptotic behavior and multiplicity results. A general existence result for all $c \in (0, \sqrt{2})$ was missing until the work \cite{Maris2}, where the case $N \geq 3$ is addressed. For the planar case $N=2$, an existence result for almost all $c \in (0, \sqrt{2})$ has been recently given, see \cite{noi}. More references can be found in the survey \cite{survey}.

In this paper we are concerned with the doubly periodic case, that is, solutions which are $T$-periodic in all variables $x_i$. This question has been addressed in \cite{jems, CMP09} as a tool to get finite energy solutions as the period goes to infinity. The approach of \cite{jems, CMP09} consists in minimizing the energy under a constraint $P(\psi) = p$. In this way the speed $c$ appears as a Lagrange multiplier and is not controlled.

It is to be noted that, as commented in \cite{CMP09}, the case of periodic solutions is interesting in its own right. The main goal of this paper is to give existence and nonexistence results of periodic traveling waves with fixed speed $c$.

In general, $T$-periodic solutions of \eqref{TWc} are stationary points of the action functional:  $$  I_T^c:H_T^1(\R^N)\rightarrow \mathbb{R}, \ I= E - c \, P$$ where $H^1_T(\R^N)$ is the Sobolev space of $T$-periodic complex-valued functions. More precisely,

\begin{equation}\label{lagrangian}
I_T^c(\psi)=\frac{1}{2}\int_{\T(T)} \vert \nabla\psi\vert^2\ dx + \frac{1}{4}\int_{\T(T)} (1-\vert\psi\vert^2)^2\ dx - \frac{c}{2} \int_{\T(T)}  (i\partial_{x_1}\psi) \cdot \psi\, dx,
\end{equation}
with $\T(T)=[0,T]^N$.

In this paper we first use variational methods to give existence of solutions. Indeed the functional $I_T^c$ attains its infimum, which is a nonconstant solution for $T$ sufficiently large. Being more specific, we can prove the following result:

\begin{theorem}\label{E} For $N=2$, $3$, and for any $c>0$ there exists $\bar{T}(c)>0$ such that, for any $T> \bar{T}(c)$, there exists a non-constant $T$-periodic solution $\bar{\psi}_T$ of \eqref{TWc}. This solution is a global minimizer for $I_T^c$, and 
$$ I_T^c(\bar{\psi}_T) <0.$$
\end{theorem}

More interestingly, one can also show the existence of a mountain pass solution. The main idea is that the constant solutions of modulus 1 form a nondegenerate curve of local minimizers if $c \in (0, \sqrt{2})$, and that the global minimizer given in Theorem \ref{E} has negative energy.

\begin{theorem}\label{MP} For $N=2$, $3$ and any $c \in (0, \sqrt{2})$ there exists $\bar{T}(c)>0$ such that, for any $T> \bar{T}(c)$, there exists a non-constant $T$-periodic solution $\tilde{\psi}_T$ of \eqref{TWc}. This solution is a mountain-pass solution for $I_T^c$, and 
	
\begin{equation} \label{eestimate} 0 <  I_T^c(\tilde{\psi}_T) \leq M(c), \end{equation}
for some $M(c)>0$ independent of $T$.
\end{theorem}

The proof of the above theorem is the original motivation of this paper. The reason is that one can conjecture that the above solution converges locally, as $T \to +\infty$, to a finite energy solution in $\R^N$. This strategy could be of use in the future to prove the existence of finite energy solutions in $\R^2$ for all $c \in (0, \sqrt{2})$, a problem that remains open in its full generality despite many attempts. In order to pass to the limit, one of the main challenges could be to find uniform bounds on the energy.

\medskip 

The above theorems have been stated for dimensions $N=2$, $3$. Under minor changes everything works also in dimension $4$; the only point there is that the nonlinear term in $I_T^c$ becomes critical in the sense of the Sobolev embeddings. For higher dimensions, $I_T^c$ is not well defined in $H^1_T(\R^N)$, and hence a truncation would be in order. For the sake of simplicity, we have prefered to restrict ourselves to the physically relevant dimensions 2 and 3.

\medskip

With those results at hand, the first question that arises naturally is whether the size requirement on the period $T$ is necessary or not. In the next theorem we show that this is indeed the case.

\begin{theorem}\label{daru}
	For all $c >0$, there exists $T^*>0$ such that for any $T \in (0, T^*)$, any $T$-periodic solution of \eqref{TWc} is necessarily constant.
\end{theorem}

The proof of the above theorem is by contradiction. If we assume the existence of $T_n$ periodic solutions $\psi_n$ with $T_n\to 0$, by uniform $L^\infty$ estimates (see \cite{farina}) and regularity arguments, one can pass to the limit in $C^k$ sense. In this way the solutions converge to a constant solution $\psi_0$. Constant solutions of \eqref{TWc} are either $0$ or a complex number of modulus one. The idea of the proof is that for $n$ sufficiently large, $\psi_n$ becomes exactly equal to its limit.

In both cases, the proof uses as a main tool the min-max characterization of the first nontrivial eigenvalue of the Laplacian. The case $\psi_0=0$ follows from a somewhat simple manipulation. The case $|\psi_0|=1$ is more delicate. First, it requires the use of a lifting, that is, to write the solution as $\psi_n= \rho_n e^{i \theta_n}$, for some functions $\rho_n : \R^N \to \R^+$, $\theta_n : \R^N \to \R$. Here it is important to realize that $\theta_n$ becomes periodic for large $n$. Finally we combine some cancellations with the Poincar\'{e} inequality for the functions $1-\rho_n$ and $\theta_n$ to conclude. 

The existence results commented above will be presented in Section 2. Section 3 is devoted to the nonexistence result given in Theorem \ref{daru}.

\section{Existence results}

In this section we will prove the existence of $T$-periodic solutions to \eqref{TWc} for large $T$. As commented in the introduction, our proof is variational, and we will consider (weak) solutions as critical points of the action functional $I_T^c: H^1_T(\R^N) \to \R$. We will denote the usual scalar product $H^1_T(\R^N)$,

$$ \langle \phi, \ \psi \rangle = \int_{\T(T)} \nabla \phi \cdot \nabla \psi + \phi \cdot \psi= \int_{\T(T)} \sum_{k=1}^N (\partial_{x_k} \phi) \cdot (\partial_{x_k} \psi) + \phi \cdot \psi, $$
where $\T(T)=[0,T]^N$. The norm is then denoted as:
$$ \| \phi \|^2 = \langle \phi,\ \phi \rangle. $$

Other norms will be denoted with a subscript.

\subsection{Proof of Theorem \ref{E}}

\begin{lemma} \label{wlsc}
	The functional $I_T^c$ is weakly lower semicontinuous for all $c,T >0$.
\end{lemma}
\begin{proof}
	Let $\lbrace \psi_n \rbrace \subset H^1(\T(T))$ be a sequence weakly convergent to some $\psi\in H^1(\T(T))$. On one hand, $\nabla\psi_n \rightharpoonup \nabla\psi$ in $L^2(\T(T))$ and by the weak lower semicontinuity,
	\[ \int _{\T(T)} \vert \nabla \psi \vert ^2 \ dx \leq \liminf_{n \to +\infty} \int _{\T(T)} \vert \nabla \psi_n \vert ^2 \ dx . \] On the other hand, by Rellich-Kondrachov theorem, there is a subsequence of $\lbrace \psi _n \rbrace$ strongly convergent to $\psi$ in $L^2(\T(T))$ and $L^4(\T(T))$. Then, up to such subsequence,
	\[ \int _{\T(T)} (1-\vert \psi \vert ^2)^2\ dx = \lim_{n \to +\infty }\int_{\T(T)} (1-\vert \psi _n\vert ^2)^2\ dx .\]
	This completes the weakly lower semicontinuity of the energy $E$. Regarding the momentum $p$, recall that $\psi_n \rightarrow \psi$ in $L^2 (\T(T))$ and $\partial_{x_1} \psi_n \rightharpoonup \partial_{x_1}\psi$ in $L^2 (\T(T))$, therefore   $ \lim p(\psi _n) = p(\psi) $ and $p$ is weakly continuous. Finally, $I_T^c = E-cp$ is weakly lower semicontinuous in view of the weakly lower semicontinuity of  $E$ and the weak continuity of  $p$. % is weakly lower semicontinuous and $p$ is weakly continuous, $I_T^c$ is also  weakly lower semicontinuous.
\end{proof}

\begin{lemma}\label{ppp10}
	The functional $I_T^c$  is coercive for all $c,T >0$.
\end{lemma} 

\begin{proof} Indeed, using the fact that for every $\lambda >0$ there exists a positive constant $K_\lambda > 0$ (depending only on $\lambda$) such that $(1-x^2)^2 \geq 4 \lambda x^2 - K_\lambda$ for all $x\in\mathbb{R}$, we have
	\[
	\begin{array}{lcll}
	E(\psi) &=& \displaystyle \frac{1}{2} \int_{\T(T)} \vert \nabla \psi \vert ^2 \ dx + \frac{1}{4} \int _{\T(T)} (1-\vert \psi \vert ^2)^2\ dx \geq \\\\ &\geq & 
	\displaystyle \frac{1}{2 }\int_{\T(T)} \vert \nabla \psi \vert ^2 \ dx +\frac{1}{4} \int _{ \T(T) } (1-\vert \psi \vert ^2)^2\ dx \geq \\\\ &\geq& 
	\displaystyle \frac{1}{2 }\int_{\T(T)} \vert \nabla \psi \vert ^2 \ dx + \lambda \int _{\T(T)} \vert \psi \vert ^2\ dx - K_\lambda = \\\\ 
	&=& \displaystyle \frac{1}{2} \Vert \nabla \psi \Vert^2_{L^2} + \lambda \Vert   \psi \Vert_{L^2}^2 - K_\lambda,
	\end{array}
	\]
	for some $\lambda >0$ to be determined later. The H\"{o}lder inequality leads to
	\[
	\begin{array}{lcll}
	p(\psi) &=& \displaystyle  \frac{1}{2} \int_{\T(T)} \langle i\partial_{x_1} \psi , \psi \rangle \ dx  \leq  \frac{1}{2} \Vert \nabla \psi \Vert_{L^2} \Vert \psi \Vert_{L^2}.
	\end{array}
	\]
	In other words,
	\[
	I_T^c (\psi)\geq \frac{1}{2} \Vert \nabla \psi \Vert^2_{L^2} + \lambda \Vert   \psi \Vert_{L^2}^2 - K_\lambda -\frac{c}{2} \Vert \nabla \psi \Vert_{L^2} \Vert \psi \Vert_{L^2}.
	\]
	We now make use of the inequality
	$$  \Vert \nabla \psi \Vert_{L^2} \Vert \psi \Vert_{L^2} \leq \frac{1}{2c} \Vert \nabla \psi \Vert_{L^2}^2 + \frac{c}{2} \Vert \psi \Vert_{L^2}^2.$$
	Combining these two inequalities, 
	
	\[
	\begin{array}{lcll}
	I_T^c (\psi)&\geq& \displaystyle  \frac{1}{2} \Vert \nabla \psi \Vert^2_{L^2} + \frac{\lambda}{2} \Vert   \psi \Vert_{L^2}^2 - K_\lambda -\frac{c}{2} \Vert \nabla \psi \Vert_{L^2} \Vert \psi \Vert_{L^2} \geq \\\\
	&=& \displaystyle  \frac{1}{4}  \Vert \nabla \psi \Vert^2_{L^2} + \left(\lambda - \frac{c^2}{4}\right) \Vert  \psi \Vert^2_{L^2} - K_\lambda .
	\end{array}
	\]
	It suffices to choose $\lambda > \frac{c^2}{4}$ to conclude.

\end{proof}

As a consequence of the two previous lemmas, the functional $I_T^c$ attains its infimum. It remains to show that the minimizer is different from a constant solution. Observe that:

$$ I_T^c(0)= \frac{1}{4} T^N, \ I_T^c(e^{i \theta})= 0 \ \mbox{ for any } \theta \in \R.$$

Next lemma, due to \cite{Maris2}, will be the key to show that $I_T^c$ may achieve negative values if $T$ is sufficiently large.

\begin{lemma} [lemma 4.4 of \cite{Maris2}]\label{4.4}
	There exists a continuous map $R \mapsto \upsilon_R$ from $[2,\infty )$ to $H^1 (\mathbb{R}^N)$ such that $\psi_R\in C_0 (\mathbb{R}^N)$ for any $R\geq 2$ and the following estimates hold:
	\begin{enumerate}
		\item[$(i)$] $\displaystyle \int_{\mathbb{R}^N} \vert \nabla \upsilon_R\vert ^2 \ dx \leq A R^{N-2} \log R,$
		%\item[$(i)$] $\displaystyle \int_{\mathbb{R}^N} \vert \nabla \psi_R\vert ^2 \ dx \leq \alpha R^{N-2} + \beta R^{N-2}\log R,$
		\item[$(ii)$] $ \begin{vmatrix} \displaystyle
		\int_{\mathbb{R}^N} (1-\vert (1+\upsilon_R)\vert^2)^2\ dx
		\end{vmatrix} \leq  B R^{N-2},$
		%\item[$(ii)$] $ \begin{vmatrix} \displaystyle
		%\int_{\mathbb{R}^N} (1-\vert \psi_R\vert^2)^2\ dx
		%\end{vmatrix} \leq CR^{N-2},$
		\item[$(iii)$] $\displaystyle \pi \omega_{1} (R-2)^{N-1} \leq P(1+\upsilon_R)\leq \pi \omega_{1} R^{N-1}.$
		%\item[$(iii)$] $\displaystyle \pi \omega_{N-1} (R-2)^{N-1} \leq p(\psi_R)\leq \pi \omega_{N-1} R^{N-1}.$
	\end{enumerate}
	where $A , B > 0$ are constants   and $\omega_1$ denotes the measure of the unit ball in $\mathbb{R}^N$.
\end{lemma}

\begin{proof}[Proof of Theorem \ref{E}] By Lemmas \ref{wlsc} and \ref{ppp10}, the functional $I_T^c$ attains its infimum. Our aim now is to show that the minimizer cannot be a constant function.
	
\medskip Take $R$ sufficiently large so that $E(1+\upsilon_R)- c P(1+\upsilon_R)<0$, where $\upsilon_R$ is given in Lemma \ref{4.4}. We now take $\bar{T} > diam( supp \, \upsilon_R)$. For any $T> \bar{T}$ we can assume that $supp\, \upsilon_R \subset \T(T)$, up to a suitable translation. We define:

\begin{equation} \label{wR} w_R \mbox{ is the extension of $\upsilon_R$ by $0$ in $\T(T)$, and periodically to $\R^N$.} \end{equation} 

In this way we obtain that $1+w_R \in H^1_T(\R^N)$ and $I_T^c(1+w_R)<0$. As a consequence, the minimum of $I_T^c$ is negative and cannot be achieved by a constant function.

\end{proof}

\subsection{Proof of Theorem \ref{MP}} 

In the previous subsection we have proved the existence of a global minimizer at a negative value of the functional. Here we will be concerned with the existence of a mountain pass solution. For this, let us define:

$$ Z = \{ e^{i \theta}, \ \theta \in [0, 2\pi] \},$$
which is a smooth curve in $H^1_T(\R^N)$ of constant solutions to \eqref{TWc}. As commented before, $I_T^c(e^{i \theta})=0$ for any $\theta \in [0, 2\pi]$. In next result we study the behavior of $I_T^c$ around $Z$: 
	
\begin{lemma} \label{Z} If $c \in (0, \sqrt{2})$ the set $Z$ is a nondegenerate curve of local minimizers of $I_T^c$. 
\end{lemma}

\begin{proof} The proof is based on the study of the second derivative $(I_T^c)''(e^{i \theta})$. Of course this operator is $0$ on the tangent space to $Z$, or, in other words,
$$ (I_T^c)''(e^{i \theta})[i e^{i \theta}]=0.$$
The proof will be concluded if we show that $(I_T^c)''(e^{i \theta})(\phi, \phi)$ is positive definite for $\phi$ orthogonal to $i e^{i \theta}$.

By phase invariance, we can restrict ourselves to the case $\theta=0$. Observe that, denoting $ u= Re\, \phi$, $v = Im \, \phi$, we have:

$$ \langle \phi, \ i \rangle = \int_{\T(T)} \phi \cdot i = \int_{\T(T)} v.$$

We now compute:

\begin{equation} \label{segunda} (I_T^c)''(1)[\phi, \phi]= \int_{\T(T)} |\nabla \phi|^2 + 2 (\phi \cdot 1 )^2 - c (i \partial_{x_1} \phi) \cdot \phi. \end{equation}

We now check that,

$$ (\phi \cdot 1)^2= u^2,$$

and, integrating by parts,

$$\int_{\T(T)}  (i \partial_{x_1} \phi) \cdot \phi= \int_{\T(T)} (\partial_{x_1} u) v - (\partial_{x_1} v)u = -2 \int_{\T(T)} (\partial_{x_1} v)u.$$

Hence,

$$ c \left | \int_{\T(T)}  (i \partial_{x_1} \phi) \cdot \phi \right | = 2c \left |  \int_{\T(T)} (\partial_{x_1} v)u \right | \leq \frac{c}{\sqrt{2}} |\nabla v|^2 + \sqrt{2} c u^2.$$

Plugging this estimate in \eqref{segunda}, we have:

$$ (I_T^c)''(1)[\phi, \phi] \geq \int_{\T(T)} |\nabla u|^2 + |\nabla v|^2 + 2 u ^2 - \frac{c}{\sqrt{2}} |\nabla v|^2 - \sqrt{2} c u^2 $$ 
$$=\int_{\T(T)} |\nabla u|^2 + \left (1- \frac{c}{\sqrt{2}} \right )|\nabla v|^2 + \left (2 - \sqrt{2} c \right)  u ^2.$$

Observe now that if $\phi$ is orthogonal to the constant function $i$, then $ \int_{\T(T)} v =0$ and the Poincar\'{e} inequality implies that:

$$ \int_{\T(T)} |\nabla v|^2 \geq c(T) \| v \|^2,$$
for some $c(T)>0$. Then,

$$ (I_T^c)''(1)[\phi, \phi] \geq \e \|\phi\|^2,$$
for some $\e>0$, concluding the proof.

\end{proof}

The above result, together with Lemma \ref{4.4}, imply the presence of a Mountain Pass geometry. Next proposition is devoted to the study of the Palais-Smale property.

\begin{proposition}\label{qqq10}
	The functional $I_T^c$ satisfies the Palais-Smale condition for any $c,T > 0$.
\end{proposition}
\begin{proof}
	Let $\psi_n$ be a Palais-Smale sequence for $I_T^c$, that, is, a sequence such that:
	
	$$ I_T^c(\psi_n) \mbox{ is bounded, } \ (I_T^c)'(\psi_n) \to 0 \mbox{ in } (H_T)^{-1} \mbox{ sense.}$$

By Lemma \ref{ppp10}, we conclude that $\psi_n$ is a bounded sequence. Up to a subsequence, we can assume that $\psi_n \weakto \psi$. Our aim now is to show strong convergence.

By the Rellich-Kondrachov Theorem we have that $\psi_n \to \psi$ in $L^2$ and $L^4$ sense. As in Lemma \ref{wlsc}, we have: 

$$ \liminf_{n \to + \infty} \int_{\T(T)} |\nabla \psi_n|^2 \geq \int_{\T(T)} |\nabla \psi|^2.$$

$$ \lim_{n \to +\infty} \int_{\T(T)} (i \partial_{x_1} \psi_n)\psi_n = \int_{\T(T)} (i \partial_{x_1} \psi)\cdot \psi,$$

Observe that,

$$ 0 \leftarrow (I_T^c)'(\psi_n)(\psi_n) = \int_{\T(T)} |\nabla \psi_n|^2 -c (i \partial_{x_1} \psi_n)\psi_n - |\psi_n|^2 + |\psi_n|^4.$$

Moreover,

$$ 0 \leftarrow (I_T^c)'(\psi_n)(\psi_n) = \int_{\T(T)} \nabla \psi_n \cdot \nabla \psi ^2 -c (i \partial_{x_1} \psi_n)\cdot \psi - \psi_n \cdot \psi + |\psi_n|^2 \psi_n \cdot \psi $$ $$ \rightarrow  \int_{\T(T)} |\nabla \psi|^2 -c (i \partial_{x_1} \psi)\psi - |\psi|^2 + |\psi|^4.$$

As a consequence we conclude that

$$ \int_{\T(T)} |\nabla \psi_n|^2 \to \int_{\T(T)} |\nabla \psi|^2,$$
which implies that $\psi_n \to \psi$ in $H^1_T(\R^N)$. The proof is completed.

\end{proof}

\begin{proof}[Proof of Theorem \ref{MP}]

By Lemma \ref{Z}, there exists $\delta_0>0$ such that, for any $\delta \in (0, \delta_0)$, there exists $\e>0$ such that $I_T^c(\psi) > \e$ for any $\psi \in \partial N(\delta)$, where
\begin{equation} \label{piu} N(\delta)= \{\psi \in H^1_T(\R^N): \ d(\psi, Z) < \delta\}.\end{equation}

Here $d(\psi, Z)= \min \{\|\psi - z \|, \ z \in Z  \}$.

Take $\bar{T}$ as given by Theorem \ref{E}, and $w_R$ as in \eqref{wR}. Clearly, $1+ w_R \notin N(\delta)$. Define:

$$ \gamma(T) =\inf_{\alpha \in\Gamma}\max_{t\in [0,1]} I_T^c (\alpha (t)),$$
where 
$$\Gamma= \{ \alpha:[0,1] \to H^1_T(\R^N) \mbox{ continuous:}\ \alpha(0)=1, \ \alpha(1) = 1 + w_R\}.$$

By \eqref{piu}, $\gamma(T) > \e >0$, whereas $I_T^c(1)=0$, $I_T^c(1+w_R)<0$. By the well-known Mountain-Pass lemma (see for instance \cite{Ambrosetti}), we conclude that there exists $\psi$ such that $(I_T^c)'(\psi)=0$, $I_T^c(\psi) = \gamma(T)$.

\medskip We only need now to show that $\gamma(T)$ is bounded in $T$. For this, take $\alpha_0 \in \Gamma$, $\alpha_0(t)= 1 + t w_R$. Observe that by the definition of $w_R$, $I(\alpha_0(t))$ is independent of $T > \bar{T}$ for any $t \in [0,1]$. If we denote:

$$M=\max_{t\in [0,1]} I_T^c (\alpha_0 (t)),$$

we conclude that $\gamma(T) \leq M$, concluding the proof.

\end{proof}

\section{All solutions are constant if $T$ is small}

This section is devoted to prove Theorem \ref{daru}. First, we state and prove a useful lemma, that can be seen as a version of the Poincar\'{e} inequality that fits perfectly in our setting.

\begin{lemma}\label{Druiz}
	Let $f:\T(T) \to \mathbb{R}$ be a measurable function satisfying $1/2 \leq f \leq 2$ on $\T(T)$. Then for all $T>0$ there exists $C_T >0$ such that  
	$$  \int_{\T(T)} \vert \nabla u (x) \vert ^2 \ dx \geq C_T \int_{\T(T)} \vert u(x)\vert^2 \ dx$$
	for any $u\in H^1 (\T(T))$ with 
	$$  \int_{\T(T)} f(x) u(x)\ dx = 0.$$
	
	Futhermore, $C_T$ does not depend on $f$ and $C_T \to  + \infty$ as $T \to 0$. 
\end{lemma}
\begin{proof} Let us define the eigenvalue

	$$ \lambda_T(f)= \inf  \left\{ \frac{\int_{\T(T)} |\nabla u|^2 }{ \int_{\T(T)} f |u|^2 }, u \in H^1_T \setminus\{0\}, \ \int_{\T(T)} f u =0 \right \}>0.$$
	
	Note that the particular case $f=1$ is related to the classical Poincar\'{e} inequality with constant denoted by $\lambda_T (1)$. From the above formula it is obvious that $\lambda_T (2) = \lambda_T(1)/2$. 
	
	We now take advantage of the monotonicity property: $$ \lambda_T (f) > \lambda_T (2). $$ Such property follows at once from the min-max characterization of $\lambda_T(f)$ (see, for instance, \cite[Chapter 11]{strauss}):
	
	$$ \lambda_T(f)= \inf_{U} \left \{ \max_{u \in U}  \left \{ \frac{\int_{\T(T)} |\nabla u|^2 }{ \int_{\T(T)} f |u|^2 }, \ u \in U \setminus\{0\} \right \} U \subset H_T^1, \ dim(U)=2 \right \}. $$
	
	As a consequence, if $\int_{\T(T)} f(x) u(x)=0$,
	
	$$  \int_{\T(T)} \vert \nabla u (x) \vert ^2 \ dx \geq \lambda_T(f)  \int_{\T(T)} f(x) \vert u(x)\vert^2 \ dx  \geq  \frac{\lambda_T (1)}{2}\int_{\T(T)} f(x) \vert u(x)\vert^2 \ dx  $$$$ \geq \frac{\lambda_T (1)}{4}\int_{\T(T)} \vert u(x)\vert^2 \ dx  . $$
	Then we can take $C_T = \lambda_T(1)  /4$. Finally, it is well known that $\lambda_T(1)$ diverges when $T$ is small, concluding the proof.
	
\end{proof}

\begin{proof}[Proof of Theorem \ref{daru}]
	Suppose that $\lbrace \psi_n\rbrace$ is a sequence of solutions to \begin{equation} \label{Tpequeno} \Delta \psi_n + ic\partial_{x_1} \psi_n + (1-\vert \psi_n \vert^2) \psi_n = 0  \mbox{ on } \T(T_n) \end{equation} with $T_n \to 0$. We aim to show that there exists $m \in \N$ such that $\psi_n$ is constant if $n \geq m$. Just by integration of the equation on $\T(T_n)$ one obtains:
	
	\begin{equation} \label{mancava} \int_{\T(T_n)} (1-|\psi_n|^2) \psi_n =0. \end{equation}
	
	This identity will be of use in what follows.
	
	\medskip 
	
	By \cite{farina}, all solutions are uniformly bounded, which implies uniform  $C^k$ bounds via elliptic estimates, for any $k \in \N$. As a consequence, $\lbrace \psi_n\rbrace$  converges to a constant function $\psi_0$ in $C^k$ sense. Such constant must solve \eqref{TWc}, hence we have two possibilities: $\psi_0 = 0$ or $\psi_0$ is a constant of modulus one.\\
	
	\textbf{Case 1: $\psi_0 = 0$.} Multiplying the equation \eqref{Tpequeno} by $\psi_n$ and integrating, yields
	$$ \int_{\T(T_n)} \vert \nabla \psi _n(x) \vert ^2  \ dx  - c \int_{\T(T_n)} \langle i\psi_n (x) ,  \partial_{x_1} \psi_n (x) \rangle\ dx - \int_{\T(T_n)} (1 - \vert \psi_n (x)\vert ^2 ) \psi_n (x) \ dx = 0.$$
	Now, we compute this useful estimate in light of Cauchy-Schwartz inequality:
	$$ \int_{\T(T_n)} \vert \langle c\psi_n (x) ,  \partial_{x_1} \psi_n (x) \rangle \vert \ dx  \leq  \frac{c^2}{2} \int_{\T(T_n)}  \vert \psi_n (x) \vert^2 \ dx + \frac{1}{2} \int_{\T(T_n)} \vert \nabla \psi_n (x)\vert ^2 \ dx.  $$ 
	This inequality allows us to write the following
	$$ \begin{array}{lcll}
	0 &=& \displaystyle  \int_{\T(T_n)} \vert \nabla \psi_n (x) \vert ^2  - (1 - \vert \psi_n (x) \vert^2) \vert \psi_n (x) \vert ^2   -  \langle ic\psi_n (x) , \partial_{x_1} \psi_n (x) \rangle   \geq \\\\
	&\geq &  \displaystyle  \int_{\T(T_n)} \vert \nabla \psi_n (x) \vert ^2   -  \vert \psi_n (x)\vert ^2  - \frac{c^2}{2} \vert \psi_n (x) \vert^2  - \frac{1}{2}  \vert \nabla \psi_n (x)\vert ^2  = \\\\
	&=& \displaystyle  \frac{1}{2}\int_{\T(T_n)} \vert \nabla \psi_n (x) \vert ^2  - (1+c^2/2)\int_{\T(T_n)} \vert \psi_n (x)\vert ^2  \geq \\\\ &\geq & \displaystyle \left( \frac{C_{T_n}}{2}  -1 - \frac{c^2}{2} \right) \int_{\T(T_n)}  \vert \psi_n (x) \vert^2, 
	\end{array} $$
	where in the last inequality we have used Lemma \ref{Druiz} applied to $f = 1 - \vert \psi_n \vert^2$, taking advantage of \eqref{mancava}. Observe now that if $T_n$ is sufficiently small, $\frac{C_{T_n}}{2}  -1 - \frac{c^2}{2}>0$, which implies that $\psi_n$ is identically equal to $0$.
	
	\medskip
	
	\textbf{Case 2: $|\psi_0| = 1$.} By the phase invariance, we can assume that $\psi_0=1$. In this case, the function $\psi_n$ extended to $\R^N$ is vortexless for large $n$, and hence there exists a lifting $\psi_n  = \rho_n e^{i\theta_n}$ with: 
	$$\rho_n : \R^N \to \R^+, \ T_n\mbox{-periodic}, \rho_n \to 1 \mbox{ in  $C^k$ sense}, $$ 
	$$\theta_n: \R^N \to \R|_{2 \pi \Z}, \ T_n\mbox{-periodic}.$$ 
	
	Observe now that since $\psi_n \to 1$ in $C^1 $ sense, the oscillation $\max \theta_n - \min \theta_n$ converges to 0. This implies that, for large $n$,
	
$$\theta_n: \R^N \to \R, \ T_n\mbox{-periodic}, \ \theta_n \to 0 \mbox{ in  $C^k$ sense}.$$ 
	
	Again by phase invariance, we can make small rotations such that:
	\begin{equation}\label{BB}
	\int_{\T(T_n)} \theta_n (x)\ dx = 0 .
	\end{equation}
	In terms of the lifting, equation (\ref{mancava}) reads as:
	\begin{equation} \label{CC} \int_ {\T(T_n) } (1 - \rho_n^2(x)) \rho _n(x) e^{i \theta_n (x)} \ dx  = 0.\end{equation}
	
With all these preliminaries, we are now ready to begin our argument. Multiplying equation \eqref{Tpequeno} by $\psi_n$ and integrating we obtain, in terms of the lifting.

	\begin{equation}\label{eqlift}
	\int_{\T(T_n)} \vert \nabla \rho_n \vert ^2 + \rho^2 \vert \nabla \theta_n \vert^2  + c (\rho_n^2 - 1)\partial_{x_1} \theta _n - (1 - \rho_n^2 ) \rho_n^2  =0. 
	\end{equation}
	
Observe that the periodicity of $\theta_n$ has been used in the above expression.

	Note that, by Cauchy-Schwartz,
	$$ \int _ {\T(T_n) } \vert c (\rho_n^2 - 1)\partial_{x_1} \theta_n   \vert  \leq \frac{1}{2} \int_ {\T(T_n) } \vert \partial_{x_1} \theta_n \vert^2  + \frac{c^2}{2} \int _ {\T(T_n) } (1 - \rho_n^2)^2 .  $$
	Taking into account \eqref{CC} we obtain:
	$$ \int _ {\T(T_n) }  (1 - \rho_n^2  ) \rho_n^2    = \underbrace{\int _ {\T(T_n) } (1 - \rho_n^2  ) ( \rho_n^2   - \rho _n    )}_{A} +  \underbrace{\int _ {\T(T_n) } (1 - \rho_n^2  ) \rho   (1-e^{i\theta_n  })}_{B}.  $$
	For sufficiently large $n$ we have that:
	$$ A = - \int _ {\T(T_n) } ( \rho_n+ \rho_n^2  ) (1-\rho _n  )^2 \   \leq 3  \int _ {\T(T_n) }(1-\rho_n  )^2  \  .$$
	In addition, using the fact that $ \vert 1 - e^{i t} \vert \leq \vert t \vert$ for all $t\in \mathbb{R}$ and Cauchy-Schwartz,
	$$ 
	\begin{array}{lcll}
	\vert B \vert &=& \displaystyle \begin{vmatrix} \displaystyle 
	\int_ {\T(T_n) }  (1-\rho_n ^2  ) \rho_n   (1 - e^{i\theta_n  } )    
	\end{vmatrix} \leq \int _ {\T(T_n) }  (\rho_n   + \rho_n^2 ) \vert 1 - \rho_n    \vert   \vert \theta_n   \vert    \leq \\\\ &\leq & \displaystyle 3 \int _ {\T(T_n) } \vert 1 - \rho_n   \vert \vert \theta_n   \vert   \leq \frac{3}{2} \int _ {\T(T_n) } \vert 1 - \rho_n   \vert^2   + \frac{3}{2} \int   _ {\T(T_n) }  \vert \theta_n   \vert ^2 .
	\end{array}  
	$$
	Using these estimates in (\ref{eqlift}),
	$$ 
	\begin{array}{lcll}
	0&=& \displaystyle  \int_{ \T(T_n)} \vert \nabla \rho _n  \vert ^2 + \rho_n^2   \vert \nabla \theta_n   \vert^2  + c (\rho_n^2   - 1)\partial_{x_1} \theta   - (1 - \rho_n^2  ) \rho^2   \geq \\\\ &\geq& \displaystyle 
	\int_{ \T(T_n)}    \vert \nabla \rho_n   \vert ^2   +  \Big(\rho_n^2  -1/2 \Big ) \vert \nabla \theta _n  \vert^2    - \frac{c^2}{2} (1 - \rho^2 )^2   \\\\ &&
	- \displaystyle  3 \int _{ \T(T_n)}   (1- \rho _n )^2   + \frac{1}{2} (1 - \rho_n  )^2   + \frac{1}{2} \vert \theta _n \vert ^2.
	\end{array} 
	$$
In sum, we obtain that for sufficiently large $n$,

\begin{equation} \label{later} 0 \geq  \displaystyle
\int_{ \T(T_n)} \vert  \nabla \rho _n  \vert ^2   - \Big (\frac{c^2+9}{2}  \Big ) (1-\rho_n )^2  +  \frac{1}{4} \vert \nabla \theta_n   \vert^2   - \frac{3}{2} \vert \theta_n   \vert^2.   \end{equation}	
	
We now plan to apply Lemma \ref{Druiz} to the functions $\theta_n$ and $(1-\rho_n)$. Actually, (\ref{BB}) allows us to use the classical Poincar\'{e} inequality to $\theta_n$, and for large $n$ we have that 

$$  \int_{\T(T_n)}   \vert \nabla \theta_n   \vert^2   \geq  7 \int_{\T(T_n)} \vert \theta_n   \vert^2.$$

With respect to $(1-\rho_n)$, observe that taking the real part of \eqref{CC} we obtain:

$$ \int_{ \T(T_n)} (1-\rho_n ) (\rho_n + \rho_n^2 ) \cos(\theta_n )=0.$$

Hence we can use Lemma \ref{Druiz} with $f=  \frac{1}{2} (\rho_n + \rho_n^2 ) \cos(\theta_n )$ to conclude that for large $n$, we have 

$$ 	\int_{ \T(T_n)} \vert  \nabla \rho _n  \vert ^2  \geq \Big (\frac{c^2+9}{2} +1 \Big ) (1-\rho_n )^2.$$

As a consequence, the inequality \eqref{later} can hold only if $\rho_n =1$, $\theta_n =0$ for all $x \in \T(T_n)$ and $n$ large enough. This finishes the proof.

\end{proof}

\end{document}